\documentclass{article}
\usepackage[utf8]{inputenc}
\usepackage{amsthm}
\usepackage{amssymb}
\usepackage{mathtools}
\usepackage{fancyhdr}
\usepackage{nameref}

\title{On a special kind of improper integral}
\author{Alexandru Bratosin}
\date{January 2019}

\newtheorem*{theorem}{Theorem}
\newtheorem{lemma}{Lemma}

\theoremstyle{definition}
\newtheorem{definition}{Definition}

\begin{document}

\maketitle

\begin{abstract}
\mbox{}\par\vspace{-\baselineskip}
\paragraph{Introduction}
In the world of mathematical analysis, many counterintuitive answers arise from the manipulation of seemingly unrelated concepts, ideas, or functions. For example, Euler showed that \(e^{i\pi} + 1 = 0\), whereas Gauss proved that the area underneath \(y = e^{-x^2}\) spanning the whole real axis is \( \sqrt{\pi} \). In this paper, we will determine the closed-form solution of the improper integral
\[
I_n = \int_{0}^{\infty} \frac{\ln{x}}{x^n+1} dx, \ \forall n \in \mathbb{R} \text{, with}\ n > 1.
\]
Determining closed-form solutions of improper integrals have real implications not only in easing the solving of similar, yet more difficult integrals, but also in speeding up numerical approximations of the answer by making them more efficient.
\paragraph{Result} Following our calculations, we derived the formula
\[
I_n = \int_{0}^{\infty} \frac{\ln{x}}{x^n+1} dx = -\frac{\pi^2}{n^2}\cot{\frac{\pi}{n}}\csc{\frac{\pi}{n}} = -\frac{d}{dn} \Bigg[ \Gamma\Big(1-\frac{1}{n}\Big) \Gamma\Big(\frac{1}{n}\Big) \Bigg].
\]
Depending on the value of \(n\), one may come up with the following intriguing identities:
\begin{align*}
I_2 & = \int_{0}^{\infty} \frac{\ln{x}}{x^2+1} dx = 0 \text{,} \\
I_3 & =\int_{0}^{\infty} \frac{\ln{x}}{x^3+1} dx = -\frac{2\pi^2}{27} \text{,} \\
I_4 & = \int_{0}^{\infty} \frac{\ln{x}}{x^4+1} dx = -\frac{\pi^2}{8\sqrt{2}} \text{,} \\
& \dots \\
\lim_{n \to \infty} I_n & = -1.
\end{align*}
\end{abstract}

\pagebreak

\section{Introduction}

To prove the theorem, we will make use of the following definitions and lemmas. Books \cite{gamma_1}-\cite{gamma_6} should help familiarize with the notions used in this document.

\begin{definition}[{\cite[p. 255]{gamma_3}}]
The gamma function \( \Gamma(z) \) is defined as the analytic continuation of the factorial to complex arguments, with
\[ \Gamma(z) = (z-1)! \]
One of its integral representations \textit{(Euler integral of the second kind)} is
\[ \Gamma(z) = \int_{0}^{\infty} t^{z - 1}e^{-t} dt \text{,} \ \forall z \in \mathbb{C} \text{, with Re} \ z > 0, \]
with the subsequent reflection formula
\[ \Gamma(1-z)\Gamma(z) = \frac{\pi}{\sin \pi z} \text{,} \ \forall z \in \mathbb{C} \setminus \mathbb{Z}. \]
\end{definition}
\noindent Due to the difficulties in analyzing the large and rapidly-increasing function \( \Gamma'(z) \), its logarithmic derivative \( \psi(z) \) is studied instead.
\begin{definition}[{\cite[p. 258]{gamma_3}}]
The digamma function \( \psi(z) \) is defined as
\[ \psi(z) = \frac{d}{dz}\ln\Gamma(z). \]
An integral representation due to Gauss is
\[ \psi(z) = \int_{0}^{\infty} \frac{e^{-t}}{t}\frac{e^{-zt}}{1-e^{-t}} dt \text{,} \ \forall z \in \mathbb{C} \text{, with Re} \ z > 0. \]
\end{definition}

\begin{definition}[{\cite[p. 260]{gamma_3}}]
The polygamma function \( \psi^{(n)}(z) \) is defined as the \(n\)-th derivative of the digamma function, i.e.
\[ \psi^{(n)}(z) = \frac{d^n}{dz^n}\psi(z) = \frac{d^{n+1}}{dz^{n+1}}\ln\Gamma(z). \]
Using Gauss's integral representation of the digamma function
\[ \psi(z) = \frac{d}{dz} \ln{\Gamma(z)} = \int_{0}^{\infty} \frac{e^{-t}}{t} - \frac{e^{-zt}}{1-e^{-t}} dt, \]
we differentiate \( \psi(z) \ n \) times with respect to \( z \) to get the polygamma function
\[ \psi^{(n)}(z) = \frac{d^n}{dz^n}\psi(z) = \frac{d^n}{dz^n} \int_{0}^{\infty} \frac{e^{-t}}{t} - \frac{e^{-zt}}{1-e^{-t}} dt \text{, } \forall n \in \mathbb{N}^{*}. \]
Using Leibniz's rule for differentiating under the integral sign, we get the integral representation of the polygamma function
\[ \psi^{(n)}(z) = \int_{0}^{\infty} \frac{\partial^n}{\partial z^n} \Bigg[ \frac{e^{-t}}{t} - \frac{e^{-zt}}{1-e^{-t}} \Bigg] dt = (-1)^{n+1} \int_{0}^{\infty} \frac{t^ne^{-zt}}{1-e^{-t}} dt. \]
\end{definition}

\begin{definition}
The trigamma function \( \psi^{(1)}(z) \) is defined as
\[ \psi^{(1)}(z) = \frac{d}{dz}\psi(z) = \int_{0}^{\infty} \frac{te^{-zt}}{1-e^{-t}} dt. \]
\end{definition}

\begin{lemma}
For all \( z \in \mathbb{C} \text{, with} \ \text{Re}\ z > 0 \text{,} \)
\[
\int_{-\infty}^{\infty} \frac{t^ne^{-zt}}{1-e^{-t}} dt = \psi^{(n)}(1-z) + (-1)^{n+1}\psi^{(n)}(z).
\]
\end{lemma}
\begin{proof}
Using the linearity of the integral, we split it into a sum of two integrals.
\begin{equation}
\int_{-\infty}^{\infty} \frac{t^ne^{-zt}}{1-e^{-t}} dt = \int_{-\infty}^{0} \frac{t^ne^{-zt}}{1-e^{-t}} dt + \int_{0}^{\infty} \frac{t^ne^{-zt}}{1-e^{-t}} dt.
\end{equation}
From Definition 3, the right-hand summand of (1) is \( (-1)^{n+1} \psi^{(n)}(z). \)
\begin{equation}
\int_{0}^{\infty} \frac{t^ne^{-zt}}{1-e^{-t}} dt = (-1)^{n+1}\psi^{(n)}(z).
\end{equation}
For the left-hand summand of (1), substitute \( u = -t \text{, with}\ du = -dt. \)
\[
\int_{-\infty}^{0} \frac{t^ne^{-zt}}{1-e^{-t}} dt = (-1)^n \int_{0}^{\infty} \frac{u^ne^{zu}}{1-e^{u}} du.
\]
Multiply the resulting integrand by \( \frac{e^{-u}}{e^{-u}}. \)
\[
\int_{-\infty}^{0} \frac{t^ne^{-zt}}{1-e^{-t}} dt = (-1)^n \int_{0}^{\infty} \frac{u^ne^{-u(1-z)}}{-1+e^{u}} du.
\]
Factor \( -1 \) from its denominator.
\[
\int_{-\infty}^{0} \frac{t^ne^{-zt}}{1-e^{-t}} dt = (-1)^{n+1}\int_{0}^{\infty} \frac{ue^{-u(1-z)}}{1-e^{-u}} du.
\]
From Definition 3, the right-hand side is \( \psi^{(n)}(1-z) \).
\begin{equation}
\int_{-\infty}^{0} \frac{t^ne^{-zt}}{1-e^{-t}} dt = \psi^{(n)}(1-z).
\end{equation}
Therefore, using the results from (2) and (3) in (1), we get
\[
\int_{-\infty}^{\infty} \frac{t^ne^{-zt}}{1-e^{-t}} dt = \psi^{(n)}(1-z) + (-1)^{n+1}\psi^{(n)}(z).
\]
\end{proof}
\pagebreak

\begin{lemma}
For all \( z \in \mathbb{C} \text{, with} \ \text{Re}\ z > 0 \text{,} \)
\begin{equation*}
\psi^{(n)}(1-z) + (-1)^{n+1}\psi^{(n)}(z) = (-1)^n\pi\frac{d^n}{dz^n}\cot{\pi z}.
\end{equation*}
\end{lemma}
\begin{proof}
Using the reflection formula of the gamma function
\[
\Gamma(1-z)\Gamma(z) = \frac{\pi}{\sin \pi z},
\]
take the natural logarithm of both sides.
\[
\ln{\Gamma(1-z)}+\ln{\Gamma(z)} = \ln{\pi} - \ln(\sin \pi z).
\]
Differentiate both sides with respect to \( z \).
\[
\frac{d}{dz}\ln\Gamma(1-z)+\frac{d}{dz}\ln\Gamma(z) = -\pi\cot(\pi z).
\]
The above is equivalent to
\[
-\psi(1-z) + \psi(z) = -\pi\cot(\pi z).
\]
Differentiate both sides \( n \) times with respect to \( z \).
\[
(-1)^{n+1}\psi^{(n)}(1-z) + \psi^{(n)}(z) = -\pi \frac{d^n}{dz^n} \cot(\pi z).
\]
\[
(-1)^{n+2}\psi^{(n)}(1-z) + (-1)\psi^{(n)}(z) = \pi \frac{d^n}{dz^n} \cot(\pi z).  \\
\]
\[
\therefore \psi^{(n)}(1-z) + (-1)^{n+1}\psi^{(n)}(z) = (-1)^n\pi \frac{d^n}{dz^n} \cot(\pi z). \\
\]
\end{proof}

\begin{lemma}
For all \( x \in \mathbb{R} \),
\[ \sec^2x - \csc^2x = -4\cot2x\csc2x. \]
\end{lemma}
\begin{proof}
Using trigonometric identities,
\begin{align*}
\sec^2x - \csc^2x & = \frac{1}{\cos^2x} -\frac{1}{\sin^2x} \\
& = \frac{\sin^2x - \cos^2x}{\sin^2x\cos^2x} \\
& = \frac{-\cos2x}{\frac{1}{4}\sin^{2}2x} \\
& = -4\cot2x\csc2x.
\end{align*}
\end{proof}

\pagebreak

\section{Main Result}

\begin{theorem}
For all \( n \in \mathbb{R} \text{, with}\ n > 1 \text{,} \)
\begin{equation}
I_n = \int_{0}^{\infty} \frac{\ln{x}}{x^n+1} dx = -\frac{\pi^2}{n^2}\cot{\frac{\pi}{n}}\csc{\frac{\pi}{n}} = -\frac{d}{dn} \Bigg[ \Gamma\Big(1-\frac{1}{n}\Big) \Gamma\Big(\frac{1}{n}\Big) \Bigg].
\end{equation}
\end{theorem}
\begin{proof}
To prove this identity, we will calculate \(I_n\) using the reflection formula of the polygamma function. First, we'll introduce the substitution
\[ u = n\ln{x}, \]
with 
\[ du = \frac{n}{x} dx. \]
Plugging these into (4) gives us
\[ I_n = \int_{0}^{\infty} \frac{\ln{x}}{x^n+1} dx = \frac{1}{n^2} \int_{-\infty}^{\infty} \frac{ue^\frac{u}{n}}{e^{u}+1} du. \]
To obtain a convenient form similar to that of the polygamma function, we will do some basic algebraic manipulation.
\begin{align*}
I_n & = \frac{1}{n^2} \int_{-\infty}^{\infty} \frac{ue^\frac{u}{n}}{e^{u}+1} du && \text{Multiply the integrand by}\ \frac{e^{-u}}{e^{-u}}. \\  
& = \frac{1}{n^2} \int_{-\infty}^{\infty} \frac{ue^{u(\frac{1}{n}-1)}}{1+e^{-u}} du  && \text{Multiply the integrand by}\ \frac{1-e^{-u}}{1-e^{-u}}. \\
& = \frac{1}{n^2} \int_{-\infty}^{\infty} \frac{ue^{u(\frac{1}{n}-1)}(1-e^{-u})}{1-e^{-2u}} du.
\end{align*}
Our original integral is now
\begin{equation}
I_n =  \frac{1}{n^2} \int_{-\infty}^{\infty} \frac{ue^{u(\frac{1}{n}-1)}(1-e^{-u})}{1-e^{-2u}} du.
\end{equation}
We will do the following substitution
\[ t = 2u, \]
with
\[ dt = 2du. \]
Plugging these into (5) yields
\begin{align*}
I_n & = \frac{1}{n^2} \int_{-\infty}^{\infty} \frac{ue^{u(\frac{1}{n}-1)}(1-e^{-u})}{1-e^{-2u}} du \\
& = \frac{1}{4n^2} \int_{-\infty}^{\infty} \frac{te^{\frac{t}{2}(\frac{1}{n}-1)}(1-e^{-\frac{t}{2}})}{1-e^{-t}} dt \\
& = \frac{1}{4n^2} \Bigg[ \int_{-\infty}^{\infty} \frac{te^{\frac{t}{2}(\frac{1}{n}-1)}}{1-e^{-t}} dt - \int_{-\infty}^{\infty} \frac{te^{\frac{t}{2}(\frac{1}{n}-2)}}{1-e^{-t}} dt \Bigg] \\
& = \frac{1}{4n^2} \Bigg[ \int_{-\infty}^{\infty} \frac{te^{t(\frac{1}{2n}-\frac{1}{2})}}{1-e^{-t}} dt - \int_{-\infty}^{\infty} \frac{te^{t(\frac{1}{2n}-1)}}{1-e^{-t}} dt \Bigg] \\
& = \frac{1}{4n^2} \Bigg[ \int_{-\infty}^{\infty} \frac{te^{-t(\frac{1}{2}-\frac{1}{2n})}}{1-e^{-t}} dt - \int_{-\infty}^{\infty} \frac{te^{-t(1-\frac{1}{2n})}}{1-e^{-t}} dt \Bigg].
\end{align*}
We end up with the form
\begin{equation}
I_n = \frac{1}{4n^2} \Bigg[ \int_{-\infty}^{\infty} \frac{te^{-t(\frac{1}{2}-\frac{1}{2n})}}{1-e^{-t}} dt - \int_{-\infty}^{\infty} \frac{te^{-t(1-\frac{1}{2n})}}{1-e^{-t}} dt \Bigg].
\end{equation}
Using Lemma 1, (6) becomes
\begin{equation}
I_n = \frac{1}{4n^2} \Bigg[ \psi^{(1)}\Big(\frac{1}{2}-\frac{1}{2n}\Big) + \psi^{(1)}\Big(\frac{1}{2} + \frac{1}{2n}\Big) - \psi^{(1)}\Big(1 - \frac{1}{2n}\Big) - \psi^{(1)}\Big(\frac{1}{2n}\Big) \Bigg].
\end{equation}
Using Lemma 2, (7) becomes
\begin{align}
I_n & = \frac{1}{4n^2} \Bigg[\pi^2 \csc^2{\Big(\frac{\pi}{2} + \frac{\pi}{2n}\Big)} - \pi^2 \csc^2{\Big(\frac{\pi}{2n}\Big)}\Bigg] \nonumber \\
& = \frac{\pi^2}{4n^2} \Bigg[\csc^2{\Big(\frac{\pi}{2} + \frac{\pi}{2n}\Big)} - \csc^2{\Big(\frac{\pi}{2n}\Big)}\Bigg] \nonumber \\
& = \frac{\pi^2}{4n^2} \Bigg[\sec^2{\Big(\frac{\pi}{2n}\Big)} - \csc^2{\Big(\frac{\pi}{2n}\Big)}\Bigg].
\end{align}
Using Lemma 3, (8) becomes
\[ I_n = -\frac{\pi^2}{n^2}\cot{\frac{\pi}{n}}\csc{\frac{\pi}{n}} =  -\frac{d}{dn} \Bigg[ \Gamma\Big(1-\frac{1}{n}\Big) \Gamma\Big(\frac{1}{n}\Big) \Bigg].  \]
\renewcommand\qedsymbol{\(\blacksquare\)}
\end{proof}

\pagebreak

\end{document}